\documentclass[11pt,oneside]{amsart}
\usepackage{bm, calc,geometry, verbatim, graphicx, amssymb, color, mathabx, mathtools, enumitem, bm, mathrsfs, amsmath, amsthm, xspace,tabularx }
\usepackage[dvipsnames,svgnames,table]{xcolor}
\usepackage[pdftex,bookmarks,colorlinks,breaklinks]{hyperref}  
\usepackage{rotating}
\usepackage{adjustbox}
\usepackage{blindtext}
\usepackage{relsize}

\usepackage{comment}

\usepackage{array}
\usepackage{tikz}

\let\oldsum\sum
\let\oldprod\prod
\renewcommand{\sum}{\oldsum\nolimits}
\renewcommand{\prod}{\oldprod\nolimits}

\hypersetup{linkcolor=blue,citecolor=red,filecolor=dullmagenta,urlcolor=darkblue} 
\newtheorem{theorem}{Theorem}[section]
\newtheorem{corollary}[theorem]{Corollary}
\newtheorem{lemma}[theorem]{Lemma}
\newtheorem{conjecture}[theorem]{Conjecture}
\newtheorem{remark}[theorem]{Remark}

\newcommand\mult{\operatorname{\textup{{\fontfamily{ptm}\selectfont mult}}}}
\newcommand\dg{\operatorname{\textup{{\fontfamily{ptm}\selectfont deg}}}}

\DeclareMathOperator{\FF}{\mathfrak{F}}

      \makeatletter
      \def\@setcopyright{}
      \def\serieslogo@{}
      \makeatother

\begin{document}
   \author{Amin  Bahmanian}
   \address{Department of Mathematics,
  Illinois State University, Normal, IL USA 61790-4520}
   \author{Sho  Suda}
   \address{Department of Mathematics, National Defense Academy of Japan, 239-8686 Japan}

\title[Sudoku Analogues of Baranyai's Theorem]{Sudoku Analogues of Baranyai's Theorem}

\begin{abstract}
Motivated by higher-dimensional generalizations of Sudoku, we study exact block-structured decompositions, algebraic characterizations, and orthogonality for Sudoku hypercubes. Let $n=\prod_{i=1}^d a_i$, let $b_i=n/a_i$, and consider the $\lambda$-fold complete $d$-uniform $d$-partite hypergraph with $d$ vertex classes of size $n$, where the $i$th class is partitioned into $a_i$ groups of size $b_i$. Given positive integers $m_1,\dots,m_k$ with $\sum_{j=1}^k m_j=\lambda n^d$, we partition the edges into color classes of sizes $m_1,\dots,m_k$ so that, in color $j$, vertex degrees and block counts are each either $\lfloor m_j/n\rfloor$ or $\lceil m_j/n\rceil$, while the multiplicity of an underlying edge is either $\lfloor m_j/n^d\rfloor$ or $\lceil m_j/n^d\rceil$. When $m_j=nr_j$, the vertex and block balances are exact, yielding block factorizations and higher-dimensional Sudoku analogues of Baranyai's theorem.

Within the same block framework, we give a Delsarte characterization of the Sudoku condition using association schemes and study mutually orthogonal Sudoku hypercubes of order $q^3$ for prime powers $q$. For block sizes $(q^3,q^2,q)$ and $(q^3,q^3,1)$, the resulting families attain a general upper bound and are best possible. For block size $(q^2,q^2,q^2)$, we construct $q^2(q^2-1)(q^2-q)$ mutually orthogonal hypercubes; this construction is asymptotically best possible as $q\to\infty$.
\end{abstract}

\subjclass[2020]{Primary 05C70; Secondary 05B15, 05C65, 05E30}  
\keywords{Sudoku, Baranyai's theorem, Latin hypercubes, orthogonality,
hypergraph factorization, block factorization, association schemes,
Delsarte designs, finite fields}
   \maketitle   

\section{Introduction}
\label{Intro}

In classical Sudoku, the objective is to fill a $9\times9$ array with
the symbols in $\{1,\dots,9\}$ so that every row, every column, and
each of the nine $3\times3$ subarrays, called \emph{blocks}, contains
every symbol exactly once. Motivated by this structure, we study
higher-dimensional arrays subject to analogous conditions on their
coordinate hyperplanes and blocks. To construct these arrays, we prove
a block-structured strengthening of Baranyai's hypergraph
decomposition theorem~\cite{MR535941}.

Let $[n]=\{1,\ldots,n\}$. A \emph{hypergraph} $G=(V,E)$ consists
of a finite set $V$ of \emph{vertices} and a multiset $E$ of subsets
of $V$, called \emph{edges}. The degree of $v\in V$, denoted by
$\deg_G(v)$, is the number of edges containing $v$, counted with
multiplicity. Throughout, all colorings are edge-colorings. A
$k$-coloring of $G$ assigns a color in $[k]$ to each edge occurrence.
For $j\in[k]$, let $G(j)$ denote the spanning subhypergraph consisting
of the edge occurrences of color $j$. The coloring is an
\emph{$(r_1,\ldots,r_k)$-factorization} if
$\deg_{G(j)}(v)=r_j$ for $v\in V$ and $j\in[k]$. When
$r_1=\cdots=r_k=r$, we call it an \emph{$r$-factorization}.

The complete $d$-uniform $d$-partite hypergraph with $d$ parts of
size $n$, denoted by $K_{n\times d}^d$, has vertex classes
$V_1,\ldots,V_d$ and one edge for every choice of one vertex from
each class. A hypergraph is \emph{almost regular} if the degrees of
any two vertices differ by at most one. A classical conjecture
attributed to Sylvester asserted that the complete $h$-uniform
hypergraph $K_n^h$ is $1$-factorable whenever $h\mid n$. Baranyai
proved this conjecture and, more generally, showed that the edges of
$K_n^h$ can be partitioned into spanning almost regular hypergraphs
of arbitrarily prescribed sizes summing to $\binom{n}{h}$
\cite{MR0416986}. He also settled the corresponding factorization
problem for complete multipartite uniform hypergraphs
\cite{MR535941}.

In the monumental work \cite{KeevashDesignsII}, Keevash developed a broad theory of hypergraph designs and decompositions encompassing numerous Baranyai-type questions and applications to Sudoku. The present paper concerns exact block-structured decompositions together with their connections to Sudoku hypercubes, association schemes, and orthogonality.

For a positive integer $\lambda$, let $\lambda G$ denote the
hypergraph obtained from $G$ by replacing every edge by $\lambda$
copies. A hypergraph is \emph{simple} if it has no repeated edges.

We first strengthen this multipartite setting by imposing a fixed Cartesian block structure on the vertex classes. For arbitrary prescribed color-class sizes, we simultaneously balance vertex degrees, the number of edges of each color in every block, and the multiplicity of each underlying edge. When the color-class sizes are multiples of $n$, the vertex and block balances become exact. Under the natural correspondence between edge-colorings of $K_{n\times d}^d$ and fillings of $d$-dimensional arrays, these factorizations yield, as a special case, higher-dimensional Sudoku hypercubes.

The same prescribed block structure also leads to two further directions. For a single hypercube, we give a Delsarte characterization of the Sudoku condition using a product of wreath-product association schemes. For families of hypercubes with the same block partition, we study orthogonality. A general counting bound is combined with finite-field constructions for three natural block structures of order $q^3$. For block sizes $(q^3,q^2,q)$ and $(q^3,q^3,1)$ the resulting families are best possible, while for $(q^2,q^2,q^2)$ the construction is asymptotically best possible.

\section{Block Decompositions and Sudoku Hypercubes}
\label{BlockSudoku}

Let
\[
n=\prod_{i\in[d]}a_i,
\qquad
b_i=\frac{n}{a_i}\quad\text{for }i\in[d].
\]
For $i\in[d]$, partition the vertex class $V_i$ of
$K_{n\times d}^d$ as
\[
V_i=\mathop{\dot\bigcup}_{s\in[a_i]}V_i^s,
\qquad
|V_i^s|=b_i.
\]
A $(b_1,\dots,b_d)$-\emph{block} is obtained by choosing one group
$V_i^{s_i}$ from each vertex class $V_i$ and taking the subhypergraph
induced by their union. There are
$\prod_{i\in[d]}a_i=n$ such blocks, each isomorphic to
$K_{b_1,\dots,b_d}^d$. We use the same terminology for the
corresponding subhypergraphs of $\lambda K_{n\times d}^d$.

We call an $(r_1,\dots,r_k)$-factorization of
$\lambda K_{n\times d}^d$ a
$(b_1,\dots,b_d)$-\emph{block $(r_1,\dots,r_k)$-factorization} if
every block contains exactly $r_j$ edges of color $j$ for
$j\in[k]$.

For a hypergraph $H$ and a subset $e\subseteq V(H)$, let
$\mult_H(e)$ denote the multiplicity of $e$ in $H$. For real $x$
and $y$, write $x\approx y$ if
$\lfloor y\rfloor\le x\le\lceil y\rceil$. The following theorem
simultaneously controls the size and vertex degrees of every color
class, the number of its edges in every prescribed block, and the
multiplicity of every underlying edge.

\subsection{The Main Block Decomposition Theorem}

\begin{theorem}
\label{Main}
Let $G=\lambda K_{n\times d}^d$, where
$n=\prod_{i\in[d]}a_i$ and $b_i=n/a_i$, and fix a
$(b_1,\dots,b_d)$-block partition as above. If
$m_1,\dots,m_k$ are positive integers satisfying
\[
\sum_{j\in[k]}m_j=\lambda n^d,
\]
then $G$ has a $k$-coloring such that
\[
\begin{aligned}
|E(G(j))|&=m_j,
&
\deg_{G(j)}(v)&\approx\frac{m_j}{n},\\
|E(G(j))\cap E(B)|&\approx\frac{m_j}{n},
&
\mult_{G(j)}(e)&\approx\frac{m_j}{n^d},
\end{aligned}
\]
for $j\in[k]$, $v\in V(G)$, $(b_1,\dots,b_d)$-blocks $B$, and
edges $e$ of $K_{n\times d}^d$. In particular, for $j\in[k]$,
$G(j)$ is simple if and only if $m_j\le n^d$.
\end{theorem}

When $m_j=nr_j$ for $j\in[k]$, the vertex and block balances
in Theorem~\ref{Main} become exact, while
\[
\mult_{G(j)}(e)\approx\frac{r_j}{n^{d-1}}
\]
for edges $e$ of $K_{n\times d}^d$. Hence
$\lambda K_{n\times d}^d$ has a simple
$(b_1,\dots,b_d)$-block $(r_1,\dots,r_k)$-factorization if and only if
\[
\sum_{j\in[k]}r_j=\lambda n^{d-1}
\qquad\text{and}\qquad
r_j\le n^{d-1}\quad\text{for }j\in[k].
\]
In particular, $\lambda K_{n\times d}^d$ has a simple
$(b_1,\dots,b_d)$-block $r$-factorization if and only if
$r\mid\lambda n^{d-1}$ and $r\le n^{d-1}$.

\subsection{Sudoku hypercubes}

A \emph{hypercube} $H$ of order $n$ and dimension $d$ on a symbol
set $S$ is a function
\[
H:[n]^d\longrightarrow S.
\]
Equivalently, it is a $d$-dimensional array indexed by $[n]^d$ in
which each cell contains exactly one symbol from $S$. Colorings of
$K_{n\times d}^d$ correspond naturally to such fillings: vertices in
$V_i$ correspond to values of the $i$th coordinate, edges correspond
to cells, and colors correspond to symbols. The prescribed partitions
of the vertex classes induce partitions of the coordinate sets, and
the $(b_1,\dots,b_d)$-blocks of $K_{n\times d}^d$ correspond exactly
to the Cartesian blocks obtained by choosing one group from each
coordinate partition. For $\ell\in[d]$, an \emph{$\ell$-layer} of
$H$ is obtained by fixing $d-\ell$ coordinates and allowing the
remaining $\ell$ coordinates to vary. A \emph{hyperplane} is a
$(d-1)$-layer.

Motivated by applications in the design of experiments
\cite{MR13113,MR34743} and higher-dimensional algebras
\cite{MR0004235}, Latin hypercubes have been studied since the
1940s. For positive integers $n,d,m$ and $0\le t\le d-m$, an
\emph{$(n,d,m,t)$ Latin hypercube} is a hypercube of order $n$ and
dimension $d$ on a symbol set of size $n^m$ such that every
$(d-t)$-layer contains each symbol exactly $n^{d-m-t}$ times. In
particular, a Latin square
of order $n$ is an $(n,2,1,1)$ Latin hypercube. The class
$(n,d,d-1,1)$ arises naturally in the design of experiments
\cite{MR13113,MR34743}, while $(n,d,1,d-1)$ Latin hypercubes have
been viewed as higher-dimensional analogues of permutations
\cite{MR3259813} and of Latin squares \cite{MR2399374}. 

For recent work on symmetric Latin cubes, see \cite{MR4665304};
for related embedding problems in the three-dimensional setting, see
\cite{MR4728465}. Huggan, Mullen, Stevens, and Thomson
\cite{MR3600882} introduced higher-dimensional Sudoku hypercubes. In
particular, they considered hypercubes of order $a^d$ with block size
$(a,\dots,a)$ and, when $a$ is a prime power, constructed such
hypercubes.

For our block-decomposition setting, we instead consider
$(n,d,d-1,1)$ Latin hypercubes. Given the prescribed coordinate
partitions above, we call an $(n,d,d-1,1)$ Latin hypercube an
\emph{$(n,d,d-1,1)$ Sudoku hypercube of block size
$(b_1,\dots,b_d)$} if every prescribed Cartesian block contains each
symbol exactly once. When the parameters are clear, we simply call it
a \emph{Sudoku hypercube}. In dimension two, this gives the usual
generalized Sudoku condition; classical Sudoku is the case $n=9$ with
block size $(3,3)$.

For a positive integer $\lambda$, a \emph{$\lambda$-fold
hypercube} of order $n$ and dimension $d$ on a symbol set $S$ assigns
to each cell of $[n]^d$ a multiset of exactly $\lambda$ symbols from
$S$. An \emph{$(n,d,m,t,\lambda)$ Latin hypercube} is a
$\lambda$-fold hypercube on a symbol set of size $n^m$ such that every
$(d-t)$-layer contains each symbol exactly $\lambda n^{d-m-t}$ times.
It is \emph{simple} if no symbol is repeated within a cell.

\begin{corollary}
Let $n=\prod_{i\in[d]}a_i$ and let $b_i=n/a_i$ for $i\in[d]$.
If $r_1,\dots,r_k$ are positive integers satisfying
\[
\sum_{j\in[k]}r_j=\lambda n^{d-1},
\]
then there exists a $\lambda$-fold hypercube of order $n$ and
dimension $d$ on $k$ symbols such that every hyperplane and every
$(b_1,\dots,b_d)$-block contains exactly $r_j$ occurrences of
symbol $j$, and the multiplicity of symbol $j$ in every cell is
either
\[
\left\lfloor\frac{r_j}{n^{d-1}}\right\rfloor
\qquad\text{or}\qquad
\left\lceil\frac{r_j}{n^{d-1}}\right\rceil
\]
for $j\in[k]$. Moreover, the $\lambda$-fold hypercube can be
chosen simple if and only if
\[
r_j\le n^{d-1}\qquad\text{for }j\in[k].
\]
\end{corollary}

Taking $m\in[d-1]$, $k=n^m$, and
$r_j=\lambda n^{d-m-1}$ for $j\in[k]$ gives an
$(n,d,m,1,\lambda)$ Latin hypercube in which every
$(b_1,\dots,b_d)$-block contains each symbol exactly
$\lambda n^{d-m-1}$ times. Moreover, the multiplicity of each symbol
in every cell is either
\[
\left\lfloor\frac{\lambda}{n^m}\right\rfloor
\qquad\text{or}\qquad
\left\lceil\frac{\lambda}{n^m}\right\rceil,
\]
and this $\lambda$-fold hypercube can be chosen simple if and only if
$\lambda\le n^m$. We regard these as generalized Sudoku-type
$\lambda$-fold hypercubes.
Figure~\ref{SudokuFig} illustrates a $(4,3,2,1)$ Sudoku hypercube with
block size $(4,2,2)$, displayed in its four parallel layers.

\begin{figure}[htbp]
\centering
\includegraphics[width=\textwidth]{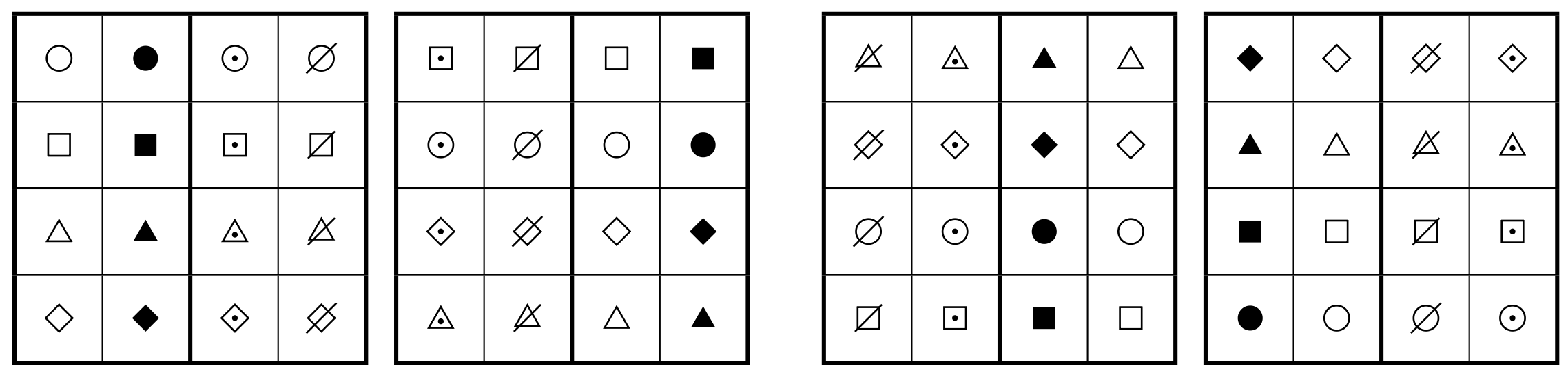}
\caption{A $(4,3,2,1)$ Sudoku hypercube with block size $(4,2,2)$.
Heavy vertical lines and the larger gap between pairs of layers indicate
the block boundaries.}
\label{SudokuFig}
\end{figure}

\subsection{Balance and a Delsarte Characterization}
\label{Delsarte}

For the remainder of this subsection, assume $d\ge2$. Let $X=[n]^d$ be the cell set. For $i\in[d]$, let
$\mathcal H_i$ be the partition of $X$ into the $n$ hyperplanes
obtained by fixing the $i$th coordinate. Each such hyperplane has
$n^{d-1}$ cells. Let $\mathcal B$ be the partition into the
prescribed Cartesian blocks. Since
$\prod_{i\in[d]}a_i=n$ and $\prod_{i\in[d]}b_i=n^{d-1}$,
the partition $\mathcal B$ also has $n$ parts, each of size
$n^{d-1}$.

Theorem~\ref{Main} has a particularly simple interpretation
when $\lambda=1$. If $m_1,\dots,m_k$ are positive integers with
$\sum_{j\in[k]}m_j=n^d$, then $X$ has a partition
$X=C_1\cup\cdots\cup C_k$, with $|C_j|=m_j$ for $j\in[k]$, such that
\[
|C_j\cap P|
\in
\left\{
\left\lfloor\frac{m_j}{n}\right\rfloor,
\left\lceil\frac{m_j}{n}\right\rceil
\right\}
\]
for $j\in[k]$ and for every coordinate hyperplane or Cartesian block
$P$. In particular, if $n\mid m_j$, then
$|C_j\cap P|=m_j/n$ for every such $P$.

The partitions above fit naturally into the theory of orthogonal
block structures and association schemes; see
Anagnostopoulou-Merkouri, Bailey, and Cameron
\cite{MR4978378}. We shall only use the association-scheme language
needed for the characterization below.

We next give an algebraic characterization of the Sudoku condition
using association schemes. We recall the terminology needed below.
Let $\Omega$ be a finite set, and let $R_0,\dots,R_r$ be symmetric
binary relations on $\Omega$ with adjacency matrices
$A_0,\dots,A_r$. These relations form a \emph{$r$-class symmetric association
scheme} if $A_0=I$, $\sum_{i=0}^r A_i=J$, and $A_iA_j$ is a
nonnegative integral linear combination of $A_0,\dots,A_r$ for
$i,j\in\{0\}\cup[r]$. Since $A_iA_j$ is a linear combination of
symmetric matrices, it is symmetric, and hence
$A_iA_j=(A_iA_j)^\top=A_jA_i$. Thus the matrices $A_i$ commute, and
since they are real symmetric, they are simultaneously orthogonally
diagonalizable.

Because the matrices $A_0,\dots,A_r$ are simultaneously
orthogonally diagonalizable, $\mathbb R^\Omega$ decomposes into their
common eigenspaces. We let $E_0,\dots,E_r$ be the orthogonal
projections onto these maximal eigenspaces. These are the
\emph{primitive idempotents} of the scheme, with
$E_0=|\Omega|^{-1}J$. If $C\subseteq\Omega$ has characteristic
vector $\chi_C$ and $T\subseteq[r]$, then $C$ is a
\emph{Delsarte $T$-design} if $E_i\chi_C=0$ for $i\in T$.
For background on association schemes and Delsarte designs, see
\cite{Bailey2004,Delsarte1973,Delsarte1976,Delsarte1977}.

We shall also use the standard direct product
$\mathfrak A\otimes\mathfrak B$ and wreath product
$\mathfrak A\wr\mathfrak B$ of association schemes. For our purposes,
it is enough to describe their primitive idempotents. If
$\mathfrak A$ and $\mathfrak B$ have primitive idempotents
$E_0,\dots,E_r$ and $F_0,\dots,F_s$, respectively, then the primitive
idempotents of $\mathfrak A\otimes\mathfrak B$ are
$E_i\otimes F_j$ for $i\in\{0\}\cup[r]$ and
$j\in\{0\}\cup[s]$. With our convention, the primitive idempotents
of $\mathfrak A\wr\mathfrak B$ are $E_0\otimes F_j$ for
$j\in\{0\}\cup[s]$, together with $E_i\otimes I$ for $i\in[r]$.

For $m\ge2$, we let $H(1,m)$ denote the one-class association scheme
on $m$ points. Its adjacency matrices and primitive idempotents are
\[
A_0^{(m)}=I_m,\qquad A_1^{(m)}=J_m-I_m,\qquad
E_0^{(m)}=\frac1mJ_m,\qquad
E_1^{(m)}=I_m-\frac1mJ_m.
\]
For $m=1$, we let $E_0^{(1)}=I_1$ and $E_1^{(1)}=0$, and regard
$H(1,1)$ as the trivial association scheme.

For $j\in[d]$, we identify the $j$th coordinate set $[n]$ with
$[b_j]\times[a_j]$ so that the prescribed groups in the $j$th
coordinate are $[b_j]\times\{s\}$ for $s\in[a_j]$. We let
\[
F_0^{(j)}
=E_0^{(b_j)}\otimes E_0^{(a_j)},\qquad
F_1^{(j)}
=E_0^{(b_j)}\otimes E_1^{(a_j)},\qquad
F_2^{(j)}
=E_1^{(b_j)}\otimes I_{a_j}.
\]
When $a_j,b_j\ge2$, these are the primitive idempotents of
$H(1,b_j)\wr H(1,a_j)$. If $a_j=1$ or $b_j=1$, $H(1,b_j)\wr H(1,a_j)=H(1,n)$ and one of
$F_1^{(j)}$ and $F_2^{(j)}$ is zero, and the nonzero matrices among
$F_0^{(j)},F_1^{(j)},F_2^{(j)}$ are the primitive idempotents.
We relabel the two nonzero primitive idempotents as $F_0^{(j)}=n^{-1}J$ and  $F_1^{(j)}$. 

These projections have a simple interpretation. The image of
$F_0^{(j)}$ consists of the constant vectors. The image of
$F_1^{(j)}$ consists of the vectors that are constant on each
prescribed group in the $j$th coordinate and have total sum zero,
while the image of $F_2^{(j)}$ consists of the vectors whose sum on
each such group is zero. Thus $F_1^{(j)}$ records variation between
the prescribed groups, whereas $F_2^{(j)}$ records variation within
them. In particular,
\[
F_1^{(j)}+F_2^{(j)}=E_1^{(n)}.
\]

We let $N=n^{d-1}$ and
\[
\mathfrak X=
(H(1,b_1)\wr H(1,a_1))\otimes\cdots\otimes
(H(1,b_d)\wr H(1,a_d))\otimes H(1,N),
\]
on the underlying set
\[
\Omega=
\prod_{j\in[d]}([b_j]\times[a_j])\times[N].
\]
Its nonzero primitive idempotents are the nonzero tensor products
\[
F_{i_1}^{(1)}\otimes\cdots\otimes F_{i_d}^{(d)}
\otimes E_i^{(N)},
\]
where $i_j\in\{0,1,2\}$ for $j\in[d]$ and $i\in\{0,1\}$.

Let $H:[n]^d\to[N]$ be a hypercube. Under the above identification
of its coordinate sets, we identify $H$ with the graph of its symbol
assignment,
\[
C=\{(x_1,\dots,x_d,H(x_1,\dots,x_d)):
x_j\in[b_j]\times[a_j]\text{ for }j\in[d]\}\subseteq\Omega,
\]
and we let $\chi$ be the characteristic vector of $C$.

We may encode $C$ as a two-dimensional table with $n^d$ rows and
$d+1$ columns. The row corresponding to the cell
$x=(x_1,\dots,x_d)$ is $(x_1,\dots,x_d,H(x))$. Thus the first $d$
columns record the coordinates of the cell and the last column records
its symbol. Each coordinate column takes $n$ possible values, while
the symbol column takes $N$ possible values; in the terminology of
orthogonal arrays, these possible values are called \emph{levels}.
Any two coordinate columns are automatically balanced, since for two
fixed coordinate values there are exactly $n^{d-2}$ cells having those
values. A coordinate column and the symbol column are balanced if and
only if every ordered pair consisting of a coordinate value and a
symbol occurs exactly once, which is equivalent to saying that every
hyperplane obtained by fixing that coordinate contains every symbol
exactly once. Hence $H$ is an $(n,d,d-1,1)$ Latin hypercube if and
only if every pair of columns in this table is balanced. Equivalently,
the table is a mixed-level orthogonal array of strength $2$.

By the Delsarte characterization of mixed-level orthogonal arrays
\cite[Example~2.5]{Martin1999}, the table above has strength $2$ if
and only if
\[
\left(
E_{\varepsilon_1}^{(n)}\otimes\cdots\otimes
E_{\varepsilon_d}^{(n)}\otimes E_{\varepsilon}^{(N)}
\right)\chi=0
\]
whenever $\varepsilon_1,\dots,\varepsilon_d,\varepsilon\in\{0,1\}$
and exactly one or two of them are equal to $1$. Since
$E_1^{(n)}=F_1^{(j)}+F_2^{(j)}$ in the $j$th coordinate, with
$F_1^{(j)}$ and $F_2^{(j)}$ orthogonal projections, this is
equivalent in $\mathfrak X$ to
\[
\left(F_{i_1}^{(1)}\otimes\cdots\otimes F_{i_d}^{(d)}
\otimes E_i^{(N)}\right)\chi=0
\]
whenever $i_j\in\{0,1,2\}$ for $j\in[d]$, $i\in\{0,1\}$, and
exactly one or two of $i_1,\dots,i_d,i$ are nonzero. 
Equivalently, $C$ is a Delsarte $T$-design, where 
\[T=\left\{(i_1,\ldots,i_d,i)\in\{0,1,2\}^d\times\{0,1\}
\mid 
1\leq \bigl|\{j\in[d] \mid i_j \neq0\}\bigr|+i\leq2
\right\}.
\]

A tensor product
containing a zero idempotent vanishes automatically.
Define $\varepsilon_j=1$ if $a_j\geq 2$, and $\varepsilon_j=0$ if $a_j= 1$. 
For each $k\in[d]$, define $\mathcal{I}_k=\{i \in\mathbb{Z} \mid 0 \leq i \leq \varepsilon_k  \}$.

\begin{theorem}
\label{DelsarteSudoku}
Let $H$ be an $(n,d,d-1,1)$ Latin hypercube, and let $\chi$ be the
characteristic vector defined above. Then $H$ is a Sudoku hypercube
if and only if
$C$ is a Delsarte $\mathcal{I}_1\times \cdots \times \mathcal{I}_d\times \{1\}$-design.
\end{theorem}
\begin{proof}
For $c=(c_1,\dots,c_d)\in \mathcal{I}_1\times \cdots \times \mathcal{I}_d$, let
\[
P_c=
F_{c_1}^{(1)}\otimes\cdots\otimes F_{c_d}^{(d)}
\otimes E_1^{(N)}.
\]
The nonzero $P_c$ are mutually orthogonal projections. Hence
$P_c\chi=0$ for $c\in \mathcal{I}_1\times \cdots \times \mathcal{I}_d$ if and only if
\[
\left(
\bigotimes_{j\in[d]}
\left(\sum_{k=0}^{\varepsilon_j}F_k^{(j)}\right)
\otimes E_1^{(N)}
\right)\chi=0.
\]
Since
\[
\sum_{k=0}^{\varepsilon_j}F_k^{(j)}
=
E_0^{(b_j)}\otimes I_{a_j}
=
\frac{1}{b_j}J_{b_j}\otimes I_{a_j},
\]
this is equivalent to
\[
\left(
\bigotimes_{j\in[d]}
(J_{b_j}\otimes I_{a_j})
\otimes E_1^{(N)}
\right)\chi=0.
\]

For $s=(s_1,\dots,s_d)\in[a_1]\times\cdots\times[a_d]$, let
$y_s\in\mathbb R^N$ be the symbol-count vector of the corresponding
Cartesian block. Expanding the preceding equation and collecting the
terms belonging to the same block gives
\[
\sum_{s\in[a_1]\times\cdots\times[a_d]}
\left(
\bigotimes_{j\in[d]}
(\mathbf 1_{b_j}\otimes e_{s_j}^{(a_j)})
\right)
\otimes E_1^{(N)}y_s=0.
\]
The vectors
\[
\bigotimes_{j\in[d]}
(\mathbf 1_{b_j}\otimes e_{s_j}^{(a_j)}),
\qquad
s\in[a_1]\times\cdots\times[a_d],
\]
are nonzero and mutually orthogonal. Hence the preceding equation
holds if and only if
\[
E_1^{(N)}y_s=0
\qquad
\text{for }s\in[a_1]\times\cdots\times[a_d].
\]

Since
\[
E_1^{(N)}=I_N-\frac1N J_N,
\]
we have $\ker E_1^{(N)}=\langle\mathbf 1_N\rangle$. Thus
$E_1^{(N)}y_s=0$ if and only if $y_s=\alpha_s\mathbf 1_N$ for some
scalar $\alpha_s$. The sum of the coordinates of $y_s$ is the number
of cells in the corresponding block, namely
\[
\prod_{j\in[d]}b_j=N.
\]
Hence $\alpha_s=1$, so $y_s=\mathbf 1_N$. Therefore each symbol
occurs exactly once in each prescribed Cartesian block. Since $H$ is
already an $(n,d,d-1,1)$ Latin hypercube, this is precisely the Sudoku
condition.
\end{proof}

Since $H$ is already a Latin hypercube, the conditions in
Theorem~\ref{DelsarteSudoku} for which at most one of
$c_1,\dots,c_d$ is nonzero are automatic. Thus the additional Sudoku
conditions are precisely those for which at least two of
$c_1,\dots,c_d$ are nonzero.

\subsection{Orthogonality}

For results on orthogonal Latin hypercubes of prime-power order, see
\cite{MR2860603}. For $(n,d,m,t)$ Latin hypercubes with $d\ge2m$,
two hypercubes are \emph{orthogonal} if, when superimposed, each
ordered pair of symbols occurs exactly $n^{d-2m}$ times; see
\cite{MR2860603}. Huggan, Mullen, Stevens, and Thomson
\cite{MR3600882} studied orthogonality for several classes of
Sudoku-like hypercubes. In particular, when $a$ is a prime power,
they constructed $(a^d-a)(a^d-a^2)\cdots(a^d-a^{d-1})$ 
mutually orthogonal hypercubes of order $a^d$ with subcubes of side
length $a$ in each coordinate.

Let $L$ and $M$ be Latin hypercubes on symbol sets $S$ and $T$,
respectively. We say that $L$ and $M$ are \emph{injectively
orthogonal} if no ordered pair $(s,t)\in S\times T$ occurs in more
than one cell when $L$ and $M$ are superimposed. We will usually say
simply that $L$ and $M$ are \emph{orthogonal}.

When $d<2m$, we use injective orthogonality for $(n,d,m,t)$ Latin
hypercubes. This is the natural analogue of the usual definition,
since $n^{d-2m}<1$ and hence equal positive multiplicities of all
ordered pairs are impossible. When $d=2m$, the two definitions agree.

Figure~\ref{OrthoPair} superimposes two orthogonal
$(4,3,2,1)$ Sudoku hypercubes with block size $(4,2,2)$. Each
cell records the two corresponding symbols, with the symbol from the
first cube written first.

\begin{figure}[htbp]
\centering
\includegraphics[width=\textwidth]{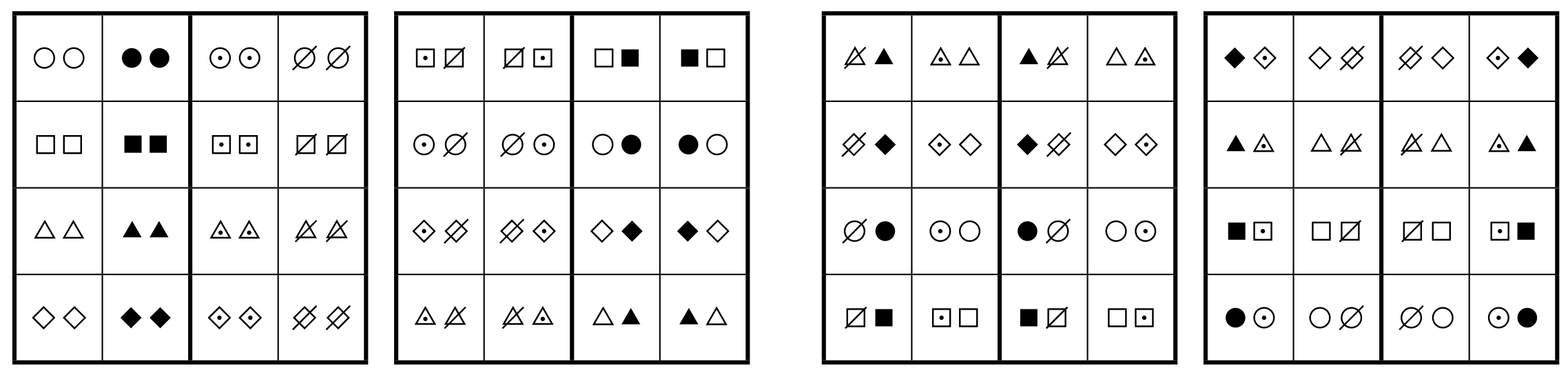}
\caption{The superposition of two orthogonal $(4,3,2,1)$ Sudoku hypercubes
with block size $(4,2,2)$. Heavy vertical lines and the larger
gap between pairs of layers indicate the block boundaries.}
\label{OrthoPair}
\end{figure}

The same block structure admits a complete set of six mutually
orthogonal $(4,3,2,1)$ Sudoku hypercubes; applying
Lemma~\ref{BlockBound} after permuting the coordinates
gives the corresponding upper bound. Figure~\ref{CompleteFam} displays their
superposition in the same four-layer format.

\begin{figure}[htbp]
\centering
\includegraphics[width=\textwidth]{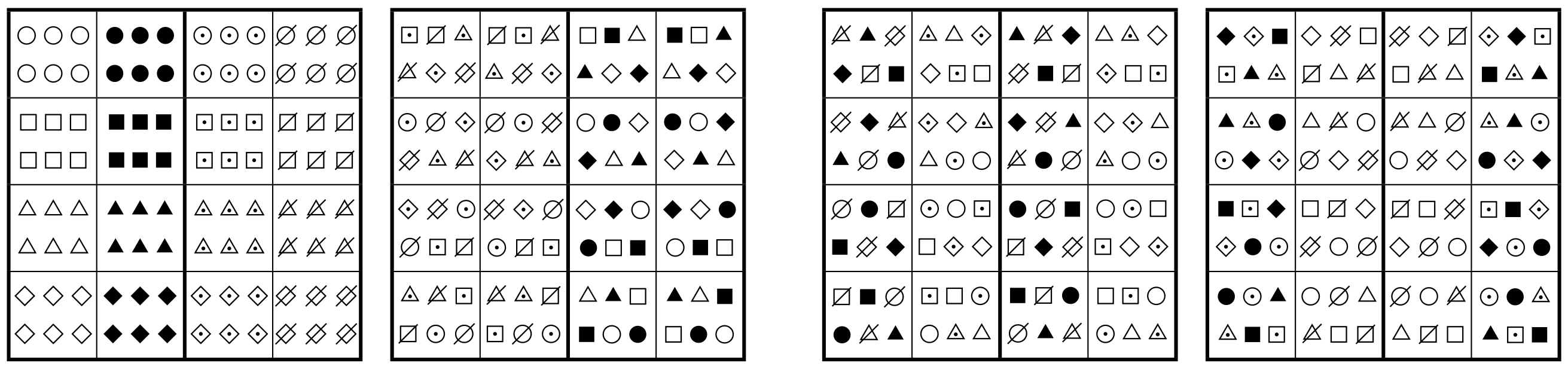}
\caption{The superposition of a complete set of six mutually
orthogonal $(4,3,2,1)$ Sudoku hypercubes with block size
$(4,2,2)$. Each cell contains the six corresponding symbols. Heavy
vertical lines and the larger gap between pairs of layers indicate
the block boundaries.}
\label{CompleteFam}
\end{figure}

In Section~\ref{FiniteField}, we use finite fields to
construct mutually orthogonal $(q^3,3,2,1)$ Sudoku hypercubes for
three natural block structures. For block sizes $(q^3,q^2,q)$ and
$(q^3,q^3,1)$, the resulting family sizes are best possible. For
block size $(q^2,q^2,q^2)$, we construct a family of size
$q^2(q^2-1)(q^2-q)$, which is asymptotically best possible as
$q\to\infty$. 

These constructions also have a hypergraph interpretation. Two $1$-factorizations of a hypergraph are \emph{orthogonal} if every
factor of one and every factor of the other have at most one edge in
common.
This notion is classical for graphs and has also been studied for
complete multipartite graphs; see, for example,
\cite{MR3942272,MR785657}. Under the correspondence between
hypergraph colorings and Latin hypercubes, the symbols of an
$(n,d,d-1,1)$ Latin hypercube correspond to the factors of a
$1$-factorization of $K_{n\times d}^d$. Hence two such Latin
hypercubes are orthogonal precisely when the corresponding
$1$-factorizations are orthogonal. Thus our finite-field
constructions yield mutually orthogonal block $1$-factorizations of
$K_{q^3\times3}^3$.

\section{Block Decompositions of Multipartite Hypergraphs}

Let $G=(V,E)$ be a hypergraph. We say that $G$ is
\emph{$d$-uniform} if $|e|=d$ for $e\in E$, and
\emph{$d$-partite} if its vertex set has a partition
$V=V_1\cup\cdots\cup V_d$ such that
$|e\cap V_i|=1$ for $e\in E$ and $i\in[d]$.
A subset $e\subseteq V$ is \emph{transversal} if
$|e\cap V_i|=1$ for $i\in[d]$. Suppose further that
$V_i=\bigcup_{s\in[a_i]}V_i^s$ is a partition for $i\in[d]$.
With respect to these partitions, a \emph{block} is the subhypergraph
of $G$ induced by choosing groups $V_i^{s_i}$ for $i\in[d]$.

Recall that $\mult_G(e)$ denotes the multiplicity of a subset
$e\subseteq V$ in $G$. We let $I(G)$ be the incidence bipartite graph
of $G$, with parts $V$ and the edge occurrences of $G$, where a
vertex $v\in V$ is adjacent to an edge occurrence if that edge
contains $v$. The edges of $I(G)$ are called the \emph{flags} of
$G$. For $v\in V$, we let $\FF_G(v)$ be the set of flags incident
with $v$, and for $e\subseteq V$ we let $\FF_G(v,e)$ be the set of
flags joining $v$ to occurrences of $e$. Thus
$|\FF_G(v,e)|=\mult_G(e)$ whenever $v\in e$. When the context is
clear, we write $\mult(e)$ and $\FF(v)$ for $\mult_G(e)$ and
$\FF_G(v)$, respectively. For $j\in[k]$, we also write
$\dg_j(v)=\deg_{G(j)}(v)$, $\mult_j(e)=\mult_{G(j)}(e)$,
$\FF_j(v)=\FF_{G(j)}(v)$, and
$\FF_j(v,e)=\FF_{G(j)}(v,e)$.

We will repeatedly use the following elementary properties of
$\approx$. If $x\approx y$, then $x/q\approx y/q$ for positive
integers $q$. If $x,y\in\mathbb Z$ and $x\approx y$, then $x=y$.
The relation $\approx$ is transitive. Finally, if $a=b-c$ and
$c\approx x$, where $a,b,c\in\mathbb Z$, then $a\approx b-x$.

A family of sets is \emph{laminar} if any two of its members are
either disjoint or one contains the other. We use the following lemma
of Nash--Williams.

\begin{lemma}\textup{(Nash--Williams \cite[Lemma~2]{MR0916377})}
\label{NW}
For a positive integer $q$ and two laminar families
$\mathscr A$ and $\mathscr B$ of subsets of a finite set $S$, there
exists $F\subseteq S$ such that
\[
|F\cap P|\approx\frac{|P|}{q}
\qquad\text{for }P\in\mathscr A\cup\mathscr B.
\]
\end{lemma}

We prove Theorem~\ref{Main} by splitting vertices in two
stages. Starting with one vertex in each part, the first stage splits
the $i$th part into $a_i$ vertices and establishes the required
balance on the prescribed blocks. The second stage splits these
vertices further into $b_i=n/a_i$ vertices while preserving the
number of edges of each color in every block. At each step, the
Nash--Williams laminar rounding lemma distributes the incident flags
as evenly as possible while controlling colored degrees and edge
multiplicities.

\subsection{Initial Splitting}

An \emph{$(\ell,d,k)$-graph} is a $k$-colored $d$-uniform
$d$-partite hypergraph on $\ell$ vertices.

\begin{lemma}
\label{FirstSplit}
Let $n,d,k,\ell,\lambda,a_1,\dots,a_d,m_1,\dots,m_k$ be positive
integers such that
\[
n=\prod_{i\in[d]}a_i,\qquad
\sum_{j\in[k]}m_j=\lambda n^d,\qquad
d\le\ell\le\sum_{i\in[d]}a_i.
\]
There exists an $(\ell,d,k)$-graph
$G=(V_1\cup\cdots\cup V_d,E)$ such that there is
$x_i\in V_i$ with $|V_i|\le a_i$ for $i\in[d]$, and
\[
|E(G(j))|=m_j,\qquad
\mult(e)=\lambda n^{d-1}\prod_{v\in e}g(v),\qquad
\frac{\dg_j(u)}{g(u)}\approx\frac{m_j}{a_i},\qquad
\frac{\mult_j(e)}{\prod_{v\in e}g(v)}\approx\frac{m_j}{n},
\]
for $i\in[d]$, $u\in V_i$, $j\in[k]$, and transversal
$e\subseteq V$, where
\[
g(x_i)=a_i-|V_i|+1\quad\text{for }i\in[d],
\qquad
g(u)=1\quad\text{otherwise}.
\]
\end{lemma}

\begin{proof}
We proceed by induction on $\ell$. For $\ell=d$, we let
$G=\lambda n^dK_{1\times d}^d$, with
$V(G)=\{x_1,\dots,x_d\}$, and color exactly $m_j$ edge occurrences
with color $j$ for $j\in[k]$. Since $g(x_i)=a_i$ for
$i\in[d]$ and $\prod_{i\in[d]}a_i=n$, we have
\[
|E(G(j))|=m_j,\qquad
\frac{\dg_j(x_i)}{g(x_i)}=\frac{m_j}{a_i},
\]
and
\[
\mult(\{x_1,\dots,x_d\})
=\lambda n^{d-1}\prod_{i\in[d]}g(x_i),\qquad
\frac{\mult_j(\{x_1,\dots,x_d\})}
{\prod_{i\in[d]}g(x_i)}
=\frac{m_j}{n}.
\]
Thus the result holds for $\ell=d$.

For the inductive step, suppose
$d\le\ell<\sum_{i\in[d]}a_i$. Without loss of generality,
$|V_1|<a_1$. We let
$q=g(x_1)=a_1-|V_1|+1\ge2$ and define
\[
\begin{aligned}
\mathscr A
&=\{\FF_j(x_1):j\in[k]\},\\
\mathscr B
&=\{\FF(x_1,S):S\subseteq V,\ x_1\in S,\ \mult(S)>0\}\\
&\quad\cup
\{\FF_j(x_1,S):S\subseteq V,\ x_1\in S,\ \mult(S)>0,\ j\in[k]\}.
\end{aligned}
\]
The members of $\mathscr A$ are pairwise disjoint. In
$\mathscr B$, flag sets belonging to distinct subsets $S$ are
disjoint, while $\FF_j(x_1,S)\subseteq\FF(x_1,S)$. Hence
$\mathscr A$ and $\mathscr B$ are laminar. By
Lemma~\ref{NW}, there exists $F\subseteq\FF(x_1)$ such that
\[
|F\cap P|\approx\frac{|P|}{q}
\qquad\text{for }P\in\mathscr A\cup\mathscr B.
\]

Add a new vertex $\alpha$ to $V_1$, transfer to $\alpha$ the flags
in $F$, and denote the resulting hypergraph by $G'$. We let
$g'(x_1)=q-1$, $g'(\alpha)=1$, and $g'(u)=g(u)$ otherwise.
Clearly $|E(G'(j))|=m_j$ for $j\in[k]$. Moreover,
\[
\begin{aligned}
\dg'_j(\alpha)
&=|F\cap\FF_j(x_1)|
\approx\frac{\dg_j(x_1)}q
=\frac{\dg_j(x_1)}{g(x_1)}
\approx\frac{m_j}{a_1},\\
\frac{\dg'_j(x_1)}{g'(x_1)}
&=\frac{\dg_j(x_1)-\dg'_j(\alpha)}{q-1}
\approx
\frac{\dg_j(x_1)-\dg_j(x_1)/q}{q-1}\\
&=\frac{\dg_j(x_1)}q
\approx\frac{m_j}{a_1}.
\end{aligned}
\]

Let $f=\{x_1\}\cup U$ be a transversal subset of $V$, and let
$e=\{\alpha\}\cup U$. By the inductive hypothesis,
$\mult(f)=\lambda n^{d-1}\prod_{v\in f}g(v)>0$, so
$\FF(x_1,f)\in\mathscr B$. Since
$\mult(f)/q=\lambda n^{d-1}\prod_{v\in U}g(v)$ is an integer, the
rounding is exact. Hence
\[
\begin{aligned}
\mult'(e)
&=|F\cap\FF(x_1,f)|
=\frac{\mult(f)}q
=\lambda n^{d-1}\prod_{v\in e}g'(v),\\
\mult'(f)
&=\mult(f)-\mult'(e)
=\lambda n^{d-1}(q-1)\prod_{v\in U}g(v)
=\lambda n^{d-1}\prod_{v\in f}g'(v).
\end{aligned}
\]

Similarly, for $j\in[k]$,
\[
\begin{aligned}
\frac{\mult'_j(e)}{\prod_{v\in e}g'(v)}
&=\frac{|F\cap\FF_j(x_1,f)|}{\prod_{v\in U}g(v)}
\approx
\frac{\mult_j(f)}{q\prod_{v\in U}g(v)}
=\frac{\mult_j(f)}{\prod_{v\in f}g(v)}
\approx\frac{m_j}{n},\\
\frac{\mult'_j(f)}{\prod_{v\in f}g'(v)}
&=\frac{\mult_j(f)-\mult'_j(e)}
{(q-1)\prod_{v\in U}g(v)}
\approx
\frac{\mult_j(f)-\mult_j(f)/q}
{(q-1)\prod_{v\in U}g(v)}\\
&=\frac{\mult_j(f)}{\prod_{v\in f}g(v)}
\approx\frac{m_j}{n}.
\end{aligned}
\]
The remaining degrees and multiplicities are unchanged, completing the
induction.
\end{proof}

\subsection{Block Refinement}

The following lemma completes the splitting construction; taking
$\ell=dn$ yields Theorem~\ref{Main}.

\begin{lemma}
\label{BlockRefine}
Let $n,d,k,\ell,\lambda,a_1,\dots,a_d,m_1,\dots,m_k$ be positive
integers such that
\[
n=\prod_{i\in[d]}a_i,\qquad
\sum_{j\in[k]}m_j=\lambda n^d,\qquad
\sum_{i\in[d]}a_i\le\ell\le dn,\qquad
b_i=\frac{n}{a_i}\quad\text{for }i\in[d].
\]
There exists an $(\ell,d,k)$-graph
$G=(V_1\cup\cdots\cup V_d,E)$, with partitions
\[
V_i=\bigcup_{s\in[a_i]}V_i^s\qquad\text{for }i\in[d],
\]
such that $x_i^s\in V_i^s$ for $i\in[d]$ and $s\in[a_i]$,
\[
a_i\le |V_i|\le n,\qquad
1\le |V_i^s|\le b_i,\qquad
|E(G(j))|=m_j,
\]
and
\[
\mult(e)=\lambda\prod_{v\in e}h(v),\qquad
\frac{\dg_j(u)}{h(u)}\approx\frac{m_j}{n},\qquad
\frac{\mult_j(e)}{\prod_{v\in e}h(v)}
\approx\frac{m_j}{n^d},
\]
for $i\in[d]$, $s\in[a_i]$, $u\in V_i^s$, $j\in[k]$, and
transversal $e\subseteq V$, where
\[
h(x_i^s)=b_i-|V_i^s|+1
\quad\text{for }i\in[d]\text{ and }s\in[a_i],\qquad
h(u)=1\quad\text{otherwise}.
\]
Moreover,
\[
|E(G(j))\cap E(B)|\approx\frac{m_j}{n}
\]
for $j\in[k]$ and blocks $B$.
\end{lemma}

\begin{proof}
We proceed by induction on $\ell$. For
$\ell=\sum_{i\in[d]}a_i$, apply Lemma~\ref{FirstSplit}.
Since $|V_i|\le a_i$ for $i\in[d]$, we have $|V_i|=a_i$. Write
\[
V_i=\{x_i^1,\dots,x_i^{a_i}\},
\qquad
V_i^s=\{x_i^s\},
\]
and let $h(x_i^s)=b_i$. Since $g\equiv1$,
Lemma~\ref{FirstSplit} gives
\[
|E(G(j))|=m_j,\qquad
\mult(e)=\lambda n^{d-1},\qquad
\dg_j(u)\approx\frac{m_j}{a_i},\qquad
\mult_j(e)\approx\frac{m_j}{n}.
\]
Since $b_i=n/a_i$ and $\prod_{i\in[d]}b_i=n^{d-1}$, it follows that
\[
\mult(e)=\lambda\prod_{v\in e}h(v),\qquad
\frac{\dg_j(u)}{h(u)}\approx\frac{m_j}{n},\qquad
\frac{\mult_j(e)}{\prod_{v\in e}h(v)}
\approx\frac{m_j}{n^d}.
\]
Finally, a block corresponds to a single transversal $e$ and therefore
contains $\mult_j(e)\approx m_j/n$ edges of color $j$.

For the inductive step, suppose
$\sum_{i\in[d]}a_i\le\ell<dn$. Some $V_i^s$ has size less than
$b_i$; without loss of generality, assume $|V_1^1|<b_1$. We let
$q=h(x_1^1)=b_1-|V_1^1|+1\ge2$ and let
\[
\begin{aligned}
\mathscr A
&=\{\FF_j(x_1^1):j\in[k]\},\\
\mathscr B
&=\{\FF(x_1^1,S):S\subseteq V,\ x_1^1\in S,\ \mult(S)>0\}\\
&\quad\cup
\{\FF_j(x_1^1,S):S\subseteq V,\ x_1^1\in S,\
\mult(S)>0,\ j\in[k]\}.
\end{aligned}
\]
The members of $\mathscr A$ are pairwise disjoint. In
$\mathscr B$, flag sets belonging to distinct subsets $S$ are
disjoint, while $\FF_j(x_1^1,S)\subseteq\FF(x_1^1,S)$. Hence
$\mathscr A$ and $\mathscr B$ are laminar. By
Lemma~\ref{NW}, there exists $F\subseteq\FF(x_1^1)$ such that
\[
|F\cap P|\approx\frac{|P|}{q}
\qquad\text{for }P\in\mathscr A\cup\mathscr B.
\]

Add a new vertex $\alpha$ to $V_1^1$, transfer to $\alpha$ the flags
in $F$, and denote the resulting hypergraph by $G'$. We let
$h'(x_1^1)=q-1$, $h'(\alpha)=1$, and $h'(u)=h(u)$ otherwise.
Then $|E(G'(j))|=m_j$ for $j\in[k]$. Since $x_1^1$ and $\alpha$
belong to the same group $V_1^1$, no edge changes blocks, so the
number of edges of each color in each block is unchanged.

For $j\in[k]$,
\[
\begin{aligned}
\dg'_j(\alpha)
&=|F\cap\FF_j(x_1^1)|
\approx\frac{\dg_j(x_1^1)}q
=\frac{\dg_j(x_1^1)}{h(x_1^1)}
\approx\frac{m_j}{n},\\
\frac{\dg'_j(x_1^1)}{h'(x_1^1)}
&=\frac{\dg_j(x_1^1)-\dg'_j(\alpha)}{q-1}
\approx\frac{\dg_j(x_1^1)-\dg_j(x_1^1)/q}{q-1}\\
&=\frac{\dg_j(x_1^1)}q
\approx\frac{m_j}{n}.
\end{aligned}
\]

Let $f=\{x_1^1\}\cup U$ be a transversal subset of $V$, and let
$e=\{\alpha\}\cup U$. By the inductive hypothesis,
$\mult(f)=\lambda\prod_{v\in f}h(v)>0$, so
$\FF(x_1^1,f)\in\mathscr B$. Since
$\mult(f)/q=\lambda\prod_{v\in U}h(v)$ is an integer, the rounding
is exact. Hence
\[
\begin{aligned}
\mult'(e)
&=|F\cap\FF(x_1^1,f)|
=\frac{\mult(f)}q
=\lambda\prod_{v\in e}h'(v),\\
\mult'(f)
&=\mult(f)-\mult'(e)
=\lambda(q-1)\prod_{v\in U}h(v)
=\lambda\prod_{v\in f}h'(v).
\end{aligned}
\]

Similarly, for $j\in[k]$,
\[
\begin{aligned}
\frac{\mult'_j(e)}{\prod_{v\in e}h'(v)}
&=\frac{|F\cap\FF_j(x_1^1,f)|}{\prod_{v\in U}h(v)}
\approx\frac{\mult_j(f)}{q\prod_{v\in U}h(v)}
=\frac{\mult_j(f)}{\prod_{v\in f}h(v)}
\approx\frac{m_j}{n^d},\\
\frac{\mult'_j(f)}{\prod_{v\in f}h'(v)}
&=\frac{\mult_j(f)-\mult'_j(e)}
{(q-1)\prod_{v\in U}h(v)}
\approx
\frac{\mult_j(f)-\mult_j(f)/q}
{(q-1)\prod_{v\in U}h(v)}\\
&=\frac{\mult_j(f)}{\prod_{v\in f}h(v)}
\approx\frac{m_j}{n^d}.
\end{aligned}
\]
The remaining degrees and multiplicities are unchanged, completing the
induction.
\end{proof}

\begin{proof}[Proof of Theorem~\ref{Main}]
Take $\ell=dn$ in Lemma~\ref{BlockRefine}. Since
$\sum_{i\in[d]}|V_i|=dn$ and $|V_i|\le n$ for $i\in[d]$, we have
$|V_i|=n$. Since $V_i$ is partitioned into $a_i$ groups of size at
most $b_i$, where $a_ib_i=n$, it follows that $|V_i^s|=b_i$ for
$i\in[d]$ and $s\in[a_i]$. Hence $h\equiv1$, so
$\mult(e)=\lambda$ for transversal $e$, and therefore
$G=\lambda K_{n\times d}^d$.

Lemma~\ref{BlockRefine} now gives
\[
\begin{aligned}
|E(G(j))|&=m_j,
&
\dg_j(v)&\approx\frac{m_j}{n},\\
|E(G(j))\cap E(B)|&\approx\frac{m_j}{n},
&
\mult_j(e)&\approx\frac{m_j}{n^d},
\end{aligned}
\]
for $j\in[k]$, vertices $v$, blocks $B$, and transversal edges $e$.

If $m_j\le n^d$, then $\mult_j(e)\in\{0,1\}$ for $e$, so $G(j)$ is
simple. Conversely, if $m_j>n^d$, then $G(j)$ has more than $n^d$
edge occurrences but $K_{n\times d}^d$ has only $n^d$ distinct
edges, so $G(j)$ is not simple.
\end{proof}

\section{Finite-Field Constructions of Orthogonal Sudoku Hypercubes}
\label{FiniteField}

We now construct mutually orthogonal Sudoku hypercubes of order
$q^3$. Recall that two such hypercubes are orthogonal if the ordered
pairs of symbols appearing in corresponding cells are all distinct.
A $(q^3,3,2,1)$ Latin hypercube has $q^6$ symbols. Partition the
positions in the three coordinates into groups of sizes
$b_1,b_2,b_3$, respectively, and let the blocks be the Cartesian
products of these groups. Since a block contains each symbol exactly
once, necessarily
\[
b_1b_2b_3=q^6.
\]
We call $(b_1,b_2,b_3)$ the block size. Throughout this section, all
hypercubes in a family use the same fixed block partition.

Here we restrict attention to the three-dimensional order-$q^3$
setting, where the finite-field constructions take a particularly
simple form.

We first give a general upper bound.

\begin{lemma}
\label{BlockBound}
Suppose $\mathcal L$ is a family of mutually orthogonal
$(n,3,2,1)$ Sudoku hypercubes for a fixed block partition
with block size $(b_1,b_2,b_3)$. If $b_1<n$, then
\[
|\mathcal L|\le b_1(b_2-1)(b_3-1).
\]
\end{lemma}

\begin{proof}
Fix a cell $x$, and choose a block $B$ which differs from the block
containing $x$ only in the first coordinate. Let $X$ be the set of
cells of $B$ which differ from $x$ in all three coordinates. Then
$|X|=b_1(b_2-1)(b_3-1)$. We count
\[
S=\{(L,y)\mid L\in\mathcal L,\ y\in B,\ L(y)=L(x)\}.
\]
For $L\in\mathcal L$, the symbol $L(x)$ occurs exactly once in $B$,
so there is exactly one choice for $y$. Hence
$|S|=|\mathcal L|$. Moreover, such a $y$ belongs to $X$, since
otherwise $x$ and $y$ would be distinct cells in the same hyperplane
receiving the same symbol under $L$, contrary to the
$(n,3,2,1)$ Latin property.

For $y\in X$, there is at most one $L\in\mathcal L$ for which
$L(y)=L(x)$. Indeed, if this held for distinct
$L,L'\in\mathcal L$, then
\[
\bigl(L(y),L'(y)\bigr)=\bigl(L(x),L'(x)\bigr),
\]
contrary to orthogonality. Therefore
\[
|\mathcal L|=|S|\le |X|=b_1(b_2-1)(b_3-1).\qedhere\]
\end{proof}
\begin{remark}
One might wonder whether applying the linear programming method to the association scheme
$\mathfrak X=
(H(1,b_1)\wr H(1,a_1))\otimes\cdots\otimes
(H(1,b_d)\wr H(1,a_d))\otimes H(1,N)^{\otimes k}
$
could yield a better upper bound on the number $k$ of mutually orthogonal Sudoku hypercubes. However, the resulting bound coincides with that of Lemma~\ref{BlockBound}.
\end{remark}

When the entries of the block size are powers of $q$, write
$(a,b,c)=(q^r,q^s,q^t)$. Since $0\le r,s,t\le3$ and
$r+s+t=6$, the possibilities, up to permutation of the coordinates,
are
\[
(2,2,2),\qquad (3,2,1),\qquad (3,3,0).
\]

\begin{theorem}
\label{FFMain}
Let $q$ be a prime power. There exist the following families of
mutually orthogonal $(q^3,3,2,1)$ Sudoku hypercubes:
\begin{enumerate}
\item[\textup{(a)}] $q^2(q^2-1)(q^2-q)$ with block size
$(q^2,q^2,q^2)$;
\item[\textup{(b)}] $q^2(q-1)(q^3-1)$ with block size
$(q^3,q^2,q)$;
\item[\textup{(c)}] $(q^3-1)^2$ with block size
$(q^3,q^3,1)$.
\end{enumerate}
The family sizes in \textup{(b)} and \textup{(c)} are best possible.
\end{theorem}

\begin{proof}
Let $K=\mathbb F_{q^3}$, and identify $\mathbb F_q$ with the
subfield of $K$ having $q$ elements. Then $[K:\mathbb F_q]=3$.
Choose $\alpha\in K\setminus\mathbb F_q$. By the tower law for field
extensions,
\[
3=[K:\mathbb F_q]
=[K:\mathbb F_q(\alpha)]
[\mathbb F_q(\alpha):\mathbb F_q].
\]
Since $\alpha\notin\mathbb F_q$, we have
$[\mathbb F_q(\alpha):\mathbb F_q]>1$, and hence this degree is $3$.
Thus $\mathbb F_q(\alpha)=K$ and $1,\alpha,\alpha^2$ is a basis of
$K$ over $\mathbb F_q$. Consequently, an element of $K$ can be
written uniquely as $z_0+z_1\alpha+z_2\alpha^2$ with
$z_0,z_1,z_2\in\mathbb F_q$.

Write the minimal polynomial of $\alpha$ over $\mathbb F_q$ as
\[
x^3+c_2x^2+c_1x+c_0,
\qquad c_0,c_1,c_2\in\mathbb F_q.
\]
Since $\alpha$ is a root,
\[
\alpha^3=-c_2\alpha^2-c_1\alpha-c_0.
\]
We let
\[
U_0=\{0\},\qquad
U_1=\mathbb F_q,\qquad
U_2=\{z_0+z_1\alpha:z_0,z_1\in\mathbb F_q\},\qquad
U_3=K.
\]
These are nested $\mathbb F_q$-subspaces of $K$ of dimensions
$0,1,2,3$, respectively, and hence $|U_i|=q^i$.

For block size $(q^r,q^s,q^t)$, partition the first, second, and
third coordinate sets into the additive cosets of
$U_r,U_s,U_t$, respectively. These groups have the required sizes.

For $u,v\in K$, define
\[
L_{u,v}(x_1,x_2,x_3)
=(x_1+ux_3,x_2+vx_3),
\]
where $x_1,x_2,x_3\in K$, and use the elements of $K^2$ as the
$q^6$ symbols.

\medskip
\noindent\emph{Latin property.}\par
\noindent
If $x_3$ is fixed, then
\[
(x_1,x_2)\longmapsto(x_1+ux_3,x_2+vx_3)
\]
is a bijection from $K^2$ to $K^2$. If $x_2$ is fixed, then the
second coordinate of the symbol, $x_2+vx_3$, determines $x_3$
uniquely if and only if $v\ne0$; once $x_3$ is known, the first
coordinate determines $x_1$. Thus a hyperplane obtained by fixing
$x_2$ contains each symbol exactly once if and only if $v\ne0$.
Similarly, fixing $x_1$ gives each symbol exactly once if and only if
$u\ne0$. Hence $L_{u,v}$ is Latin if and only if $u\ne0$ and
$v\ne0$.

\medskip
\noindent\emph{Orthogonality.}\par
\noindent
Suppose $(u,v)\ne(u',v')$, and suppose two cells
$(x_1,x_2,x_3)$ and $(x'_1,x'_2,x'_3)$ receive the same ordered
pair of symbols under $L_{u,v}$ and $L_{u',v'}$. We let
$\delta_i=x_i-x'_i$. Then
\[
\delta_1+u\delta_3=\delta_2+v\delta_3=0,
\qquad
\delta_1+u'\delta_3=\delta_2+v'\delta_3=0.
\]
Subtracting gives
$(u-u')\delta_3=(v-v')\delta_3=0$. Since
$(u,v)\ne(u',v')$, at least one of $u-u'$ and $v-v'$ is nonzero,
so $\delta_3=0$. Hence $\delta_1=\delta_2=0$, and the two cells are
equal. Thus distinct hypercubes $L_{u,v}$ are orthogonal.

\medskip
\noindent\emph{Block condition.}\par
\noindent
A block is a translate of $U_r\times U_s\times U_t$, where
$r+s+t=6$. Since a block and the symbol set $K^2$ both have $q^6$
elements, the Sudoku condition is equivalent to the restriction of
$L_{u,v}$ to a block being injective. Since $L_{u,v}$ is linear, two
cells in the same block receive the same symbol precisely when their
difference $(y_1,y_2,y_3)\in U_r\times U_s\times U_t$ lies in the
kernel of $L_{u,v}$. Thus
\[
y_1=-uy_3,\qquad y_2=-vy_3.
\]
If $y_3=0$, then $y_1=y_2=0$. Hence the block condition holds if and
only if there is no nonzero $y_3\in U_t$ such that
\[
uy_3\in U_r
\qquad\text{and}\qquad
vy_3\in U_s.
\]

\begin{enumerate}

\item[\textup{(a)}]
Suppose the block size is $(q^2,q^2,q^2)$. The block condition is
equivalent to the nonexistence of a nonzero $y_3\in U_2$ such that
$uy_3,vy_3\in U_2$.

Write
\[
u=u_0+u_1\alpha+u_2\alpha^2,\qquad
v=v_0+v_1\alpha+v_2\alpha^2,
\]
and let $y_3=a_0+a_1\alpha\in U_2$. Using
$\alpha^3=-c_2\alpha^2-c_1\alpha-c_0$, the coefficient of
$\alpha^2$ in $uy_3$ is
\[
u_2a_0+(u_1-c_2u_2)a_1,
\]
and similarly the coefficient of $\alpha^2$ in $vy_3$ is
$v_2a_0+(v_1-c_2v_2)a_1$. Thus $uy_3,vy_3\in U_2$ precisely when
\[
\begin{pmatrix}
u_2 & u_1-c_2u_2\\
v_2 & v_1-c_2v_2
\end{pmatrix}
\begin{pmatrix}
a_0\\
a_1
\end{pmatrix}
=
\begin{pmatrix}
0\\
0
\end{pmatrix}.
\]
This system has a nonzero solution if and only if its determinant
vanishes. Since
\[
u_2(v_1-c_2v_2)-v_2(u_1-c_2u_2)=u_2v_1-u_1v_2,
\]
the block condition holds if and only if
$u_1v_2-u_2v_1\ne0$.

This says that $(u_1,u_2)$ and $(v_1,v_2)$ are linearly independent
in $\mathbb F_q^2$. There are $q^2-1$ choices for the first vector
and $q^2-q$ choices for the second outside its span. The coefficients
$u_0$ and $v_0$ are arbitrary, giving $q^2$ further choices. Hence
there are
\[
q^2(q^2-1)(q^2-q)
\]
valid pairs $(u,v)$. Linear independence also implies $u\ne0$ and
$v\ne0$, so the Latin condition is automatic. This proves
\textup{(a)}.

\item[\textup{(b)}]
For block size $(q^3,q^2,q)$, the block condition asks whether there
is a nonzero $y_3\in U_1=\mathbb F_q$ such that
$uy_3\in K$ and $vy_3\in U_2$. The first condition is automatic.
Since $y_3\in\mathbb F_q^\times$ and $U_2$ is an
$\mathbb F_q$-subspace,
\[
vy_3\in U_2
\quad\Longleftrightarrow\quad
v\in U_2.
\]
Hence the block condition is $v\notin U_2$.

There are $q^3-q^2=q^2(q-1)$ choices for $v\notin U_2$, and the
Latin condition additionally requires $u\ne0$, giving $q^3-1$
choices for $u$. Thus there are
\[
q^2(q-1)(q^3-1)
\]
valid pairs $(u,v)$. After permuting the first two coordinates,
Lemma~\ref{BlockBound} gives the same number as an
upper bound. Hence this family is best possible, proving
\textup{(b)}.

\item[\textup{(c)}]
For block size $(q^3,q^3,1)$, we have $U_t=U_0=\{0\}$, so there is
no nonzero $y_3\in U_t$. Thus the block condition is automatic.
Hence $(u,v)\in K^\times\times K^\times$ gives a valid hypercube,
yielding
\[
(q^3-1)^2
\]
mutually orthogonal Sudoku hypercubes. After permuting the coordinates
so that the side of length $1$ is first,
Lemma~\ref{BlockBound} gives $(q^3-1)^2$ as an upper
bound. Hence this family is best possible, proving \textup{(c)}.

\end{enumerate}
\end{proof}

\section{Open Problems}
\label{OpenProblems}

A partition of a $d$-dimensional array of order $n$ into
$n$ regions, all containing $n^{d-1}$ cells, is called a
\emph{Gerechte framework}. It is \emph{realizable} if there exists
an $(n,d,d-1,1)$ Latin hypercube in which every region contains each
symbol exactly once. A framework together with such a realization is
a \emph{Gerechte design}. Classical Sudoku is the case $d=2$ and
$n=9$, with the nine $3\times3$ blocks as the regions. For
background on two-dimensional Gerechte designs, see
\cite{sudokugerechtehamming}.

Courtiel and Vaughan~\cite{Courtiel_2011} proved that every
two-dimensional Gerechte framework whose regions are $s\times t$
or $t\times s$ rectangles is realizable. This suggests the following
higher-dimensional analogue.

\begin{conjecture}
\label{MixedGerechte}
Let $a_1,\dots,a_d$ be positive integers, let
$n=\prod_{i\in[d]}a_i$, and let $b_i=n/a_i$ for $i\in[d]$.
Suppose $R_1,\dots,R_n$ is a Gerechte framework such that, for
$m\in[n]$,
\[
R_m=A_{m,1}\times\cdots\times A_{m,d},
\]
where $A_{m,i}\subseteq[n]$ and, for some permutation $\pi_m$ of
$[d]$,
\[
|A_{m,i}|=b_{\pi_m(i)}
\qquad\text{for }i\in[d].
\]
Then the framework is realizable.
\end{conjecture}

Since $\prod_{i\in[d]}b_i=n^{d-1}$, the regions have the required
number of cells. Theorem~\ref{Main} settles the conjecture
when the regions arise from one fixed system of coordinate
partitions. The open case allows the coordinate subsets defining the
regions to vary from region to region.

Our second problem concerns orthogonality. For block sizes
$(q^3,q^2,q)$ and $(q^3,q^3,1)$,
Theorem~\ref{FFMain} determines the exact maximum size of a
mutually orthogonal family. For the symmetric block size
$(q^2,q^2,q^2)$, the exact maximum remains open. Our construction
gives a family of size $q^2(q^2-1)(q^2-q)$, while
Lemma~\ref{BlockBound} gives the upper bound
\[
q^2(q^2-1)^2.
\]
Thus the maximum lies between these two quantities. We conjecture
that the block bound is sharp in the symmetric case.

\begin{conjecture}
\label{SymmetricOrtho}
For a prime power $q$, the maximum size of a family of mutually
orthogonal Sudoku hypercubes of order $q^3$ and block size
$(q^2,q^2,q^2)$ is
\[
q^2(q^2-1)^2.
\]
\end{conjecture}

Our construction falls short of the conjectured maximum by the factor
$(q+1)/q=1+1/q$, and is therefore asymptotically optimal as
$q\to\infty$.

Finally, a connected version of Baranyai's theorem
\cite{BahCCA26} suggests a connected analogue of
Theorem~\ref{Main}. Since $K_{n\times d}^d$ has $dn$
vertices and its edges have size $d$, a connected spanning
subhypergraph must have at least
\[
\left\lceil\frac{dn-1}{d-1}\right\rceil
\]
edges. We conjecture that this necessary condition is also sufficient.

\begin{conjecture}
\label{ConnectedBaranyai}
Let $d\ge2$ and let $G=\lambda K_{n\times d}^d$, where
$n=\prod_{i\in[d]}a_i$ and $b_i=n/a_i$, with a fixed
$(b_1,\dots,b_d)$-block partition. If $m_1,\dots,m_k$ are positive
integers satisfying
\[
\sum_{j\in[k]}m_j=\lambda n^d,
\]
then the coloring in Theorem~\ref{Main} can be chosen so that
$G(j)$ is connected whenever
\[
m_j\ge
\left\lceil\frac{dn-1}{d-1}\right\rceil.
\]
\end{conjecture}

\subsubsection*{Acknowledgments}
Amin Bahmanian's research is partially supported by a  Faculty Research Award at Illinois State University. Sho Suda's research is supported by JSPS KAKENHI Grant Number 22K03410, 26K06904.

\bibliographystyle{plain}
\bibliography{sudoku3dbib}

@article {MR785657,
    AUTHOR = {Dinitz, Jeffrey H.},
     TITLE = {Orthogonal one-factorization graphs},
   JOURNAL = {J. Graph Theory},
  FJOURNAL = {Journal of Graph Theory},
    VOLUME = {9},
      YEAR = {1985},
    NUMBER = {1},
     PAGES = {147--159},
      ISSN = {0364-9024,1097-0118},
   MRCLASS = {05C70 (05B15)},
  MRNUMBER = {785657},
MRREVIEWER = {W.\ D.\ Wallis},
       DOI = {10.1002/jgt.3190090112},
       URL = {https://doi.org/10.1002/jgt.3190090112},
}

@article {MR3942272,
    AUTHOR = {Meszka, Mariusz and Tyniec, Magdalena},
     TITLE = {Orthogonal one-factorizations of complete multipartite graphs},
   JOURNAL = {Des. Codes Cryptogr.},
  FJOURNAL = {Designs, Codes and Cryptography. An International Journal},
    VOLUME = {87},
      YEAR = {2019},
    NUMBER = {5},
     PAGES = {987--993},
      ISSN = {0925-1022,1573-7586},
   MRCLASS = {05C70 (05B15)},
  MRNUMBER = {3942272},
MRREVIEWER = {Emine\ \c Sule\ Yaz\i c\i},
       DOI = {10.1007/s10623-018-0504-3},
       URL = {https://doi.org/10.1007/s10623-018-0504-3},
}

@article {MR0916377,
    AUTHOR = {Nash-Williams, C. St. J. A.},
     TITLE = {Amalgamations of almost regular edge-colourings of simple
              graphs},
   JOURNAL = {J. Combin. Theory Ser. B},
  FJOURNAL = {Journal of Combinatorial Theory. Series B},
    VOLUME = {43},
      YEAR = {1987},
    NUMBER = {3},
     PAGES = {322--342},
      ISSN = {0095-8956,1096-0902},
   MRCLASS = {05C75},
  MRNUMBER = {916377},
MRREVIEWER = {H.\ Joseph\ Straight},
       DOI = {10.1016/0095-8956(87)90008-6},
       URL = {https://doi.org/10.1016/0095-8956(87)90008-6},
}

@article {MR2399374,
    AUTHOR = {McKay, Brendan D. and Wanless, Ian M.},
     TITLE = {A census of small {L}atin hypercubes},
   JOURNAL = {SIAM J. Discrete Math.},
  FJOURNAL = {SIAM Journal on Discrete Mathematics},
    VOLUME = {22},
      YEAR = {2008},
    NUMBER = {2},
     PAGES = {719--736},
      ISSN = {0895-4801},
   MRCLASS = {05B15 (05A15 20N05 20N15 94B25)},
  MRNUMBER = {2399374},
MRREVIEWER = {R. M. Falc\'{o}n},
       DOI = {10.1137/070693874},
       URL = {https://doi.org/10.1137/070693874},
}

@article {MR4665304,
    AUTHOR = {Bahmanian, Amin},
     TITLE = {Symmetric layer-rainbow colorations of cubes},
   JOURNAL = {SIAM J. Discrete Math.},
  FJOURNAL = {SIAM Journal on Discrete Mathematics},
    VOLUME = {37},
      YEAR = {2023},
    NUMBER = {4},
     PAGES = {2617--2625},
      ISSN = {0895-4801,1095-7146},
   MRCLASS = {05B15 (05C15 05C65 05C70)},
  MRNUMBER = {4665304},
MRREVIEWER = {Ying\ Miao},
       DOI = {10.1137/22M1494488},
       URL = {https://doi.org/10.1137/22M1494488},
}

@article {MR535941,
    AUTHOR = {Baranyai, Zsolt},
     TITLE = {The edge-coloring of complete hypergraphs. {I}},
   JOURNAL = {J. Combin. Theory Ser. B},
  FJOURNAL = {Journal of Combinatorial Theory. Series B},
    VOLUME = {26},
      YEAR = {1979},
    NUMBER = {3},
     PAGES = {276--294},
      ISSN = {0095-8956,1096-0902},
   MRCLASS = {05C65 (15A36)},
  MRNUMBER = {535941},
MRREVIEWER = {C.\ St.\ J. A. Nash-Williams},
       DOI = {10.1016/0095-8956(79)90002-9},
       URL = {https://doi.org/10.1016/0095-8956(79)90002-9},
}

@article {BahCCA26,
	author = {Bahmanian, Amin},
	date = {2026/09/03},
	doi = {10.1007/s00493-026-00224-z},
	id = {Bahmanian2026},
	isbn = {1439-6912},
	journal = {Combinatorica},
	number = {5},
	pages = {30},
	title = {Connected Fair Detachments of Hypergraphs {I}},
	url = {https://doi.org/10.1007/s00493-026-00224-z},
	volume = {46},
	year = {2026}}

@misc{KeevashDesignsII,
  author       = {Peter Keevash},
  title        = {The Existence of Designs II},
  year         = {2018},
  eprint       = {1802.05900},
  archivePrefix= {arXiv},
  primaryClass = {math.CO}
}

@book{Bailey2004,
  author    = {Bailey, R. A.},
  title     = {Association Schemes: Designed Experiments, Algebra and Combinatorics},
  series    = {Cambridge Studies in Advanced Mathematics},
  volume    = {84},
  publisher = {Cambridge University Press},
  address   = {Cambridge},
  year      = {2004}
}

@incollection {MR0416986,
    AUTHOR = {Baranyai, Zsolt},
     TITLE = {On the factorization of the complete uniform hypergraph},
 BOOKTITLE = {Infinite and finite sets ({C}olloq., {K}eszthely, 1973;
              dedicated to {P}. {E}rd{\H o}s on his 60th birthday), {V}ol.
              {I}},
     PAGES = {91--108. Colloq. Math. Soc. J\'an\=os Bolyai, Vol. 10},
 PUBLISHER = {North-Holland, Amsterdam},
      YEAR = {1975},
   MRCLASS = {05C99},
  MRNUMBER = {0416986 (54 \#5047)},
MRREVIEWER = {D. L. Greenwell},
}

@article {MR4728465,
    AUTHOR = {Bahmanian, Amin},
     TITLE = {Toward a three-dimensional counterpart of {C}ruse's theorem},
   JOURNAL = {Proc. Amer. Math. Soc.},
  FJOURNAL = {Proceedings of the American Mathematical Society},
    VOLUME = {152},
      YEAR = {2024},
    NUMBER = {5},
     PAGES = {1947--1959},
      ISSN = {0002-9939,1088-6826},
   MRCLASS = {05B15 (05C15 05C65 05C70)},
  MRNUMBER = {4728465},
MRREVIEWER = {Alison\ M.\ Marr},
       DOI = {10.1090/proc/16714},
       URL = {https://doi.org/10.1090/proc/16714},
}

@article {MR13113,
    AUTHOR = {Fisher, R. A.},
     TITLE = {A system of confounding for factors with more than two
              alternatives, giving completely orthogonal cubes and higher
              powers},
   JOURNAL = {Ann. Eugenics},
  FJOURNAL = {Annals of Eugenics. A Journal Devoted to the Genetic Study of
              Human Populations},
    VOLUME = {12},
      YEAR = {1945},
     PAGES = {283--290},
      ISSN = {2050-1420},
   MRCLASS = {09.0X},
  MRNUMBER = {13113},
MRREVIEWER = {H. S. M. Coxeter},
}

@article {MR0004235,
    AUTHOR = {Richardson, A. R.},
     TITLE = {Algebra of {$s$}-dimensions},
   JOURNAL = {Proc. London Math. Soc. (2)},
  FJOURNAL = {Proceedings of the London Mathematical Society. Second Series},
    VOLUME = {47},
      YEAR = {1940},
     PAGES = {38--59},
      ISSN = {0024-6115},
   MRCLASS = {09.1X},
  MRNUMBER = {4235},
MRREVIEWER = {O.\ Ore},
       DOI = {10.1112/plms/s2-47.1.38},
       URL = {https://doi.org/10.1112/plms/s2-47.1.38},
}

@article {MR34743,
    AUTHOR = {Kishen, K.},
     TITLE = {On the construction of latin and hyper-graeco-latin cubes and
              hypercubes},
   JOURNAL = {J. Indian Soc. Agric. Statist.},
  FJOURNAL = {Journal of the Indian Society of Agricultural Statistics},
    VOLUME = {2},
      YEAR = {1949},
     PAGES = {20--48},
      ISSN = {0019-6363},
   MRCLASS = {09.0X},
  MRNUMBER = {34743},
MRREVIEWER = {H.\ B.\ Mann},
}

@article {MR3259813,
    AUTHOR = {Linial, Nathan and Luria, Zur},
     TITLE = {An upper bound on the number of high-dimensional permutations},
   JOURNAL = {Combinatorica},
  FJOURNAL = {Combinatorica. An International Journal on Combinatorics and
              the Theory of Computing},
    VOLUME = {34},
      YEAR = {2014},
    NUMBER = {4},
     PAGES = {471--486},
      ISSN = {0209-9683,1439-6912},
   MRCLASS = {05A05 (05A16)},
  MRNUMBER = {3259813},
MRREVIEWER = {Arnold\ Knopfmacher},
       DOI = {10.1007/s00493-011-2842-8},
       URL = {https://doi.org/10.1007/s00493-011-2842-8},
}

@article {MR2860603,
    AUTHOR = {Ethier, John T. and Mullen, Gary L. and Panario, Daniel and
              Stevens, Brett and Thomson, David},
     TITLE = {Sets of orthogonal hypercubes of class {$r$}},
   JOURNAL = {J. Combin. Theory Ser. A},
  FJOURNAL = {Journal of Combinatorial Theory. Series A},
    VOLUME = {119},
      YEAR = {2012},
    NUMBER = {2},
     PAGES = {430--439},
      ISSN = {0097-3165,1096-0899},
   MRCLASS = {05B15},
  MRNUMBER = {2860603},
MRREVIEWER = {Ilene\ H.\ Morgan},
       DOI = {10.1016/j.jcta.2011.10.001},
       URL = {https://doi.org/10.1016/j.jcta.2011.10.001},
}

@article {MR3600882,
    AUTHOR = {Huggan, M. and Mullen, G. L. and Stevens, B. and Thomson, D.},
     TITLE = {Sudoku-like arrays, codes and orthogonality},
   JOURNAL = {Des. Codes Cryptogr.},
  FJOURNAL = {Designs, Codes and Cryptography. An International Journal},
    VOLUME = {82},
      YEAR = {2017},
    NUMBER = {3},
     PAGES = {675--693},
      ISSN = {0925-1022,1573-7586},
   MRCLASS = {94B25 (05B15 05B20 05B25)},
  MRNUMBER = {3600882},
MRREVIEWER = {Ying\ Miao},
       DOI = {10.1007/s10623-016-0190-y},
       URL = {https://doi.org/10.1007/s10623-016-0190-y},
}

@article {MR4978378,
    AUTHOR = {Anagnostopoulou-Merkouri, Marina and Bailey, R. A. and
              Cameron, Peter J.},
     TITLE = {Permutation groups, partition lattices and block structures},
   JOURNAL = {Forum Math. Sigma},
  FJOURNAL = {Forum of Mathematics. Sigma},
    VOLUME = {13},
      YEAR = {2025},
     PAGES = {Paper No. e180, 32},
      ISSN = {2050-5094},
   MRCLASS = {20B05 (06B99 62K10)},
  MRNUMBER = {4978378},
MRREVIEWER = {Enoch\ Suleiman},
       DOI = {10.1017/fms.2025.10126},
       URL = {https://doi.org/10.1017/fms.2025.10126},
}

@article{Courtiel_2011,
   title={Gerechte designs with rectangular regions},
   volume={20},
   ISSN={1520-6610},
   url={http://dx.doi.org/10.1002/jcd.20300},
   DOI={10.1002/jcd.20300},
   number={2},
   journal={Journal of Combinatorial Designs},
   publisher={Wiley},
   author={Courtiel, J. and Vaughan, E. R.},
   year={2011},
   month=sep, pages={112–123} }

@article{sudokugerechtehamming,
author = {R. A. Bailey, Peter J. Cameron and Robert Connelly},
title = {Sudoku, Gerechte Designs, Resolutions, Affine Space, Spreads, Reguli, and Hamming Codes},
journal = {The American Mathematical Monthly},
volume = {115},
number = {5},
pages = {383-404},
year = {2008},
publisher = {Taylor & Francis},
doi = {10.1080/00029890.2008.11920542},
URL = { 
        https://doi.org/10.1080/00029890.2008.11920542
},
eprint = { 
        https://doi.org/10.1080/00029890.2008.11920542
}
}

@article{Delsarte1973,
  author  = {Delsarte, Philippe},
  title   = {An algebraic approach to the association schemes of coding theory},
  journal = {Philips Res. Rep. Suppl.},
  volume  = {10},
  year    = {1973},
  pages   = {vi+97}
}

@article{Delsarte1976,
  author  = {Delsarte, Philippe},
  title   = {Association schemes and $t$-designs in regular semilattices},
  journal = {J. Combin. Theory Ser. A},
  volume  = {20},
  year    = {1976},
  pages   = {230--243},
  doi     = {10.1016/0097-3165(76)90017-0}
}

@article{Delsarte1977,
  author  = {Delsarte, Philippe},
  title   = {Pairs of vectors in the space of an association scheme},
  journal = {Philips Res. Rep.},
  volume  = {32},
  year    = {1977},
  pages   = {373--411}
}

@article{Martin1999,
  author    = {William J. Martin},
  title     = {Designs in Product Association Schemes},
  journal   = {Designs, Codes and Cryptography},
  volume    = {16},
  number    = {3},
  pages     = {271--289},
  year      = {1999},
  doi       = {10.1023/A:1008399515082},
  url       = {https://springer.com}
}

\end{document}